\documentclass[10pt]{amsart}
\usepackage{amsmath}
\usepackage[utf8]{inputenc}
\usepackage[english]{babel}
\usepackage{paralist}
\usepackage{graphics}
\usepackage{epsfig}
\usepackage{graphicx}
\usepackage{epstopdf}
\usepackage[colorlinks=true]{hyperref}
\usepackage{cases}
\usepackage{prettyref}
\hypersetup{urlcolor=blue, citecolor=red}

\newtheorem{thm}{Theora}[section]
\newtheorem{Theo}[thm]{Theorem}

\newtheorem{Lem}[thm]{Lemma}
\newtheorem{Prop}[thm]{Proposition}

\newtheorem{Rem1}[thm]{Definition}
\newenvironment{proof1}{\smallskip\noindent{\it Proof}\rm}
{\hfill $\Box$\medskip}

\newcommand{\R}{\mathbb{R}}
\newcommand{\C}{\mathbb{C}}
\newcommand{\N}{\mathbb{N}}

\renewcommand\({\left(}
\renewcommand\){\right)}
\renewcommand\[{\left[}
\renewcommand\]{\right]}

\newcommand\Dim{{\rm{dim}}}
\newcommand\la{\lambda}
\newcommand\ga{\gamma}

\newcommand\si{\sigma}
\newcommand{\be}{\begin{equation}}
\newcommand{\ee}{\end{equation}}
\newcommand{\ba}{\begin{array}}
\newcommand{\ea}{\end{array}}
\newcommand{\bea}{\begin{eqnarray*}}
\newcommand{\eea}{\end{eqnarray*}}
\newcommand{\bean}{\begin{eqnarray}}
\newcommand{\eean}{\end{eqnarray}}

\newcommand\se{\sigma}
\begin{document}
\thispagestyle{empty}

\title[
Biharmonic Schrödinger equation with variable coefficients]{Exact boundary controllability of the linear Biharmonic Schrödinger equation with variable coefficients}

\author{Ka\"{\i}s Ammari}
\address{UR Analysis and Control of PDEs, UR13E564, Department of Mathematics, Faculty of Sciences of
Monastir, University of Monastir, 5019 Monastir, Tunisia}
\email{kais.ammari@fsm.rnu.tn}

\author{Hedi Bouzidi}
\address{UR Analysis and Control of PDEs, UR13E564, Department of Mathematics, Faculty of Sciences of
Monastir, University of Monastir, 5019 Monastir, Tunisia}
\email{hedi.bouzidi@fst.utm.tn}

\date{}

\begin{abstract}
In this paper, we study the  exact boundary controllability of the linear fourth-order Schr\"odinger equation, with variable physical parameters and clamped boundary conditions on a bounded interval. The control acts on the first spatial derivative at the left endpoint. We prove that  this control system is exactly controllable at any time
$T>0$. The proofs are based on a detailed spectral analysis and on the use of nonharmonic Fourier series.
\end{abstract}


\maketitle

\tableofcontents
\vfill\break
\section{Introduction}
The fourth-order cubic nonlinear Schrödinger equation or biharmonic cubic non-linear Schrödinger equation reads is given by \be\label{bihar1} i\partial_t y+
\partial_{x}^4y- \partial_{x}^2y-\mu |y|^2y=0, \ee
where $y$ is a complex-valued function and $\mu$ is a real constant. This
equation has been modeled by Karpman \cite{Karpman} and Karpman and Shagalov \cite{Karpman1} in order to describe the propagation of intense laser beams in a bulk medium with Kerr nonlinearity when small fourth-order dispersion are taken into account. The fourth-order cubic nonlinear Schrödinger equation \eqref{bihar1} could be found in various areas of physics, such that  nonlinear optics, plasma physics, superconductivity and quantum mechanics, we refer to the book of Fibich \cite{Fibich}, see also \cite{Ben, Capistrano,Paus}. 

\medskip

The well-posedness and the
dynamic properties of the biharmonic Schrödinger equation \eqref{bihar1} have been extensively studied
from the mathematical perspective, see the paper by Pausader \cite{Paus1}, see the papers by Capistrano-Filho et al.
\cite{Capistrano2, Capistrano3}, also \cite{CuiGuo, Wen1} and references therein.

\medskip

In this work we are interested in studying the controllability properties of
the linear biharmonic Schrödinger equation \eqref{bihar1} for $\mu=0$, with variable physical parameters on the bounded interval $\(0,\ell\),~\ell>0$. More precisely, we consider the following control system 
\bean
\label{bihar2} \left\{
 \begin{array}{ll}
i\rho(x)\partial_ty=-\partial_x^2\(\sigma(x)\partial_x^2y\)+
\partial_x(q(x)\partial_xy)_x, \, (t,x)\in(0,T)\times(0,\ell),\\
y(t, 0) =\partial_xy(t,0) =y(t, \ell) = 0,~~\partial_xy(t, \ell) = f(t), \in t\in(0,T),\\
y(0, x) = y^{0}(x), \,  x\in(0,\ell),
\end{array}
 \right.
\eean
where $f$ is a control that acts at the left end $x = \ell$, and
the functions $y^0$ is the initial condition. Throughout the paper, we assume the following assumptions on  the coefficients: \be \rho,\sigma \in
H^{2}(0,\ell),~q\in H^{1}(0,\ell), \label{bihar3}\ee and there exist
constants $\rho_0,~\sigma_0>0$, such that \be
\rho(x)\geq\rho_{0},~~\se(x)\geq\se_{0},~~ q(x)\geq0,~~x \in
\[0, \ell\]. \label{bihar4}\ee
For system \eqref{bihar2}, the appropriate control notion to study is the
exact  controllability, which is defined as follows: system \eqref{bihar2} is said to be exact controllable
in time $T>0$ if, given any initial state $y^0$, there exists control $f$ such that the corresponding
solution $y = y(t,x)$  satisfies $y(T,.) = 0$.

\medskip

Let us now describe the existing results on  stabilization and control of the Biharmonic Schrödinger  system \eqref{bihar2}. When $\se\equiv0$, we recover the classical second order Schrödinger equation with variable coefficients occupying the interval $\(0,\ell\)$. In this context, the stabilization of the second order Schrödinger equation been thoroughly studied, see for instance \cite{Ammari2015,Ammari20170,Ammari202101,Ammari202102}. We also refer to \cite{Ammari20121bc,Ammari20121bcb,Ammari2019,Ammari20150} for related results on exact controllability of the second order Schrödinger equation, see also  \cite{L,Hansen}, and references therein. The first result on exact controllability of the linear biharmonic Schrödinger equation \eqref{bihar1} for $\mu=0$  on a bounded domain $\Omega\subset\R^n,~n\geq1,$ has been established by Zheng and Zhongcheng \cite{Zheng}. In that paper, the authors proved that the linearized system
\be\label{bihar5}\begin{cases}
 i\partial_t y+
\Delta^2y=0,&(t,x)\in(0,T)\times\Omega, \\
y=0,~\dfrac{\partial y}{\partial\nu}=f\chi_{\Gamma_0},&(t,x)\in(0,T)\times\partial\Omega\\
y(0, x) = y^{0},& x\in\Omega,
\end{cases}\ee
is exactly controllable for any positive
time $T,$  where the control $f\in L^2\(\(0,T\)\times\Gamma_0\)$ and ${\Gamma_0}\subset\partial\Omega$. Their proof uses the Hilbert Uniqueness Method "Lions'{\rm HUM}"
(cf. Lions \cite{J.L1, J.L2}) and the multiplier
techniques \cite{V.K}. Later, Wen et al.\cite{Wen1} proved the well-posedness and the exact controllability for the linear fourth order Schr\"{o}dinger system \eqref{bihar5} with the boundary observation
$$z(t,x)=-i\Delta\(\(\Delta^2\)^{-1}y(t,x)\),~(t,x)\in\(0,T\)\times\Gamma_0.$$
As consequence, they established the exponential stability of the closed-loop system under the output feedback $f=-kz$ for any $k>0$. The same authors in \cite{Wen2}, extended these results to the case of a linear
fourth-order multi-dimensional Schrödinger equation with hinged boundary by either moment or Dirichlet boundary control and collocated observation, respectively. 

\medskip

The inverse problem of retrieving a stationary potential
from boundary measurements for the one-dimensional linear system \eqref{bihar2} with $\rho\equiv\se\equiv1 \hbox{ and }  f\equiv0,$  was studied by Zheng \cite{Zheng1}. To this end, the author proved a global Carleman estimate for the corresponding fourth order operator. Exact controllability result has been established recently by Gao \cite{Peng} when the linear system \eqref{bihar2} with $\rho\equiv\se\equiv1 \hbox{ and } q \equiv0,$ has a particular structure. In that reference, the author consider a forward and backward stochastic fourth order Schr\"{o}dinger equation and, again, uses Carleman inequalities for the adjoint problem for proving the exact controllability result. More recently, the global stabilization and exact controllability properties have been studied by Capistrano-Filho et al. \cite{Capistrano} for the  biharmonic cubic non-linear Schrödinger equation \eqref{bihar1} on a periodic domain $\mathbb{T}$ with internal control supported on an arbitrary sub-domain of $\mathbb{T}$.  More precisely, by means of some properties of propagation of compactness and regularity in Bourgain spaces, first they showed that the system is globally exponentially stabilizable. Then they used this with a local controllability result to get the global controllability for the associated control system. In particular, for the proof of the local controllability result, they combined a perturbation argument and  the fixed point theorem of Picard.

\medskip

To our knowledge, the exact controllability of the
fourth order Schr\"{o}dinger equation with variable
coefficients is still unknown. In this paper we prove that the linear  control system \eqref{bihar2} is exactly controllable in any time $T>0$, where the control $f\in L^2(0,T)$ and the initial condition $y^{0}\in H^{-2}(0,\ell)$. Our approach is essentially based on the qualitative theory of fourth-order linear differential equations, and on a precise asymptotic analysis of the eigenvalue
and eigenfunction. Firstly, we prove that all the
eigenvalues $\({\lambda_{n}}\)_{n\in\N^*}$ associated to
the control system \eqref{bihar2} with $f(t)\equiv0$ are allegorically simple. Moreover, we show that the second derivative
of each eigenfunction $\phi_n,~n\in\N^*,$ associated with the uncontrolled
system does not vanish at the end $x=\ell$. Secondly,  by a
precise computation of the asymptotics of the eigenvalues $\({\lambda_{n}}\)_{n\in\N^*}$,  we establish that the spectral gap
$$|{\la_{n+1}}-{\la_n}|\asymp {n^3}\(\frac{\pi}{\gamma}\)^{4},~\hbox{ as } n\to\infty, \gamma := \int^\ell_0
\sqrt[4]{\frac{\rho(t)}{\sigma(t)}} dt.$$ 
As a result of the
theory of non-harmonic Fourier series and a variant of Ingham's inequality due to Beurling (e.g., \cite{L}), we derive the following observability inequality  \be
\int_{0}^{T}|\partial_x^2\tilde{y}(t,\ell)|^{2}dt \asymp
\|\tilde{y}^0\|_{H^2_0(0,\ell)}^{2},
\label{bihar6}\ee
for any $T>0$, where $\tilde y$ is the solution of system
\eqref{bihar2} without control. Finally, we apply the Lions'{\rm
HUM} to deduce the
exact controllability result for the system \eqref{bihar2}.

\medskip

The rest of the paper is divided as follow: In
Section $2$, we establish the well-posedness of system
\eqref{bihar2} without control. In Section $3$, we prove the
simplicity of all the eigenvalues $\(\la_n\)_{n\geq1}$ and we
determinate the asymptotics of the associated spectral gap. In
Section $4$, we prove the observability inequality \eqref{bihar6}. Finally in Section $5$, we prove the exact controllability result for the
linear control problem \eqref{bihar2}.
\section{Well-posedness of the uncontrolled system}
In this section, we will see how solutions of system
\eqref{bihar2} without control can be developed in terms of Fourier series.
As a consequence, we establish  the existence and the uniqueness of solutions of the uncontrolled system  \eqref{bihar2} with $f(t)\equiv0$. To this end, we consider the following system
\bean
\label{bihar7} \left\{
 \begin{array}{ll}
i\rho(x)\partial_ty=-\partial_x^2\(\sigma(x)\partial_x^2y\)+
\partial_x(q(x)\partial_xy)_x,&(t,x)\in(0,T)\times(0,\ell),\\
y(t, 0) =\partial_xy(t,0) =y(t, \ell) =\partial_xy(t, \ell) =0,&t\in(0,T),\\
y(0, x) = y^{0},& x\in(0,\ell),
\end{array}
 \right.
\eean
First of all, let us define by $L^2_\rho(0,\ell)$ the space of
functions $y$ such that $$\int_0^\ell |y(x)|^{2}\rho(x)dx<\infty.$$ Throughout this paper, we
denote by $H^k(0,\ell)$ the $L^2_\rho(0,\ell)-$based Sobolev spaces for
$k> 0$. We consider the following Sobolev space $$ H^{2}_0(0,\ell):=\left\{y\in H^{2}(0,\ell)~:~y(0)=y'(0)=y(\ell)=y'(\ell)=0\right\} $$ endowed with the norm $$
\|u\|_{H^{2}_0(0,\ell)}=\|u''\|_{L^2_\rho(0,\ell)}.$$
It is easy to show by Rellich's theorem (e.g., \cite{V.K}) that the space $H^{2}_0(0,\ell)$ is densely and compactly embedded in the space
$L^2_\rho(0,\ell)$. In the sequel, we introduce the operator
$\mathcal{A}$ defined in $L^2_\rho(0,\ell)$ by setting:
 $$ \mathcal{A}y = \rho^{-1}\(\(\sigma y''\)''-\(qy'\)'\), $$
 on the domain
 $$ \mathcal D\(\mathcal{A}\)= H^{4}(0,\ell) \cap H^2_0(0,\ell),$$
which is dense in $L^2_{\rho}(0,\ell)$.
\begin{Lem}\label{rr}
The linear operator $\mathcal{A}$ is positive and self-adjoint such
that $\mathcal{A}^{-1}$ is compact. Moreover, the spectrum of
$\mathcal{A}$ is discrete and consists of a
 sequence of positive eigenvalues
$(\lambda_{n})_{n\in\mathbb{N}^{*}}$ tending to $+\infty$:
$$0<\lambda_{1}\leq\lambda_{2}\leq.......\leq\lambda_{n}\leq.....
\underset{n\rightarrow +\infty}{\longrightarrow}+\infty.$$
The corresponding eigenfunctions $(\Phi_{n})_{n\in\N^*}$ can be chosen
to form an orthonormal basis in $L^2_\rho(0,\ell)$.
\end{Lem}
\begin{proof}
Let $y \in\mathcal D\(\mathcal{A}\)$, then by integration by parts, we
have \bea \langle \mathcal{A}y,
y\rangle_{L^2_\rho(0,\ell)}&=&\int_{0}^{\ell}\Big{(}(\sigma(x)y''(x))''-(q(x)y'(x))'\Big{)} \overline{y}(x) dx \nonumber\\
&=&\int_{0}^{\ell}\sigma(x)|y''(x)|^{2}dx+q(x)|y'(x)|^{2}dx.
\eea Since $\sigma>0$ and $q\geq0$, then $$\langle
\mathcal{A}y, y\rangle_{L^2_\rho(0,\ell)}>0 \hbox{ for }y\not\equiv0,$$ and hence the quadratic
form has a positive real values, which implies that the linear operator
$\mathcal{A}$ is symmetric. Furthermore, it is easy to show that
\\$Ran(\mathcal{A}-iId)=L^2_\rho(0,\ell)$, and this means that
$\mathcal{A}$ is selfadjoint. Since the space $H^{2}_{0}(0,\ell)$ is continuously and compactly embedded in the space
$L^2_\rho(0,\ell)$, then $\mathcal{A}^{-1}$ is compact in
$L^2_\rho(0,\ell)$. The lemma is proved.
\end{proof}
Now, we give a characterization of some fractional powers of the linear
operator $\mathcal{A}$ which will be useful to give a description of
the solutions of problem \eqref{bihar7} in terms of Fourier
series. According to Lemma \ref{rr}, the operator $\mathcal{A}$ is
positive and self-adjoint, and hence it generates a scale of
interpolation spaces $\mathcal{H}_{\theta}$, $\theta \in
\mathbb{R}$. For $\theta\geq0$, the space $\mathcal{H}_{\theta}$
coincides with $\mathcal D(\mathcal{A}^{\theta})$ and is equipped
with the norm $\|u\|_\theta^2=\langle \mathcal{A}^\theta u,
\mathcal{A}^\theta u\rangle_{L^2_\rho(0,\ell)}$, and for $\theta< 0$ it
is defined as the completion of $L^2_\rho(0,\ell)$ with respect to this
norm. Furthermore, we have the following spectral representation of
space $\mathcal{H}_{\theta}$,
\be \mathcal{H}_{\theta}=
\left\{u(x)=\sum\limits_{n\in\N^*}c_n\Phi_n(x)~:
 ~\|u\|_{\theta}^2=\sum\limits_{n\in\N^*}\la_n^{2\theta}|c_n|^{2}<\infty\right\},\label{bihar8}\ee
where $\theta\in \R$, and the eigenfunctions $\(\Phi_{n}\)_{n\in\N^*}$
are defined in Lemma \ref{rr}. In particular, $$\mathcal{H}_{0}=L^2_\rho(0,\ell) \hbox{ and } \mathcal{H}_{1/2}=H^{2}_{0}(0,\ell).$$
Obviously, the solutions of problem \eqref{bihar7} can be written as $${y}(t,x)=\sum\limits_{n\in\N^*}c_ne^{i\la_nt}{\Phi}_n(x)$$
where the Fourier coefficients are given by $$c_n:=\int_{0}^{\ell} {y}^0 (x) \overline{{\Phi}_n} (x)\rho(x)dx,~n\in\N^*,$$
and $\(c_n\)\in\ell^2\(\N^*\)$. Let us denote by $\mathcal{E}_{\theta}$ the energy associated to the space $\mathcal{H}_{\theta}$, then \begin{align*}
                            \mathcal{E}_{\theta}(t) &= \|y\|_{\theta}^2=\sum\limits_{n\in \N^*}
{\la}_n^{2\theta}|c_{n}e^{i\la_nt}|^2\\&= \sum\limits_{n\in \N^*}
{\la}_n^{2\theta}|c_{n}|^2 =\mathcal{E}_{\theta}(0),
                          \end{align*}
which establishes the conservation of energy along time. As consequence,  we have the following existence and uniqueness result for problem \eqref{bihar7}.
\begin{Prop}\label{pos}
Let $\theta\in\R$ and ${y}^0\in\mathcal{H}_{\theta}$. Then problem \eqref{bihar7} has a
unique solution ${y}\in C([0,T],\mathcal{H}_{\theta})$ and is given by the
following Fourier series \be
{y}(t,x)=\sum\limits_{n\in\N^*}c_ne^{i\la_nt}{\Phi}_n(x),\label{bihar9}\ee
where ${y}^0=\sum\limits_{n\in\N^*} c_n\Phi_n$. Moreover, the energy of the system \eqref{bihar7} is conserved along the time.
\end{Prop}
\section{Spectral analysis}\label{Spe}
In this section, we investigate the main properties of all the eigenvalues $(\la_{n})_{n\in\mathbb{N}^{*}}$ of the operator $\mathcal{A}$. On one hand, we prove that all the eigenvalues $(\la_n)_{n\in\N^*}$ are algebraically simple, and then,  the second derivatives of the corresponding eigenfunctions $(\Phi_{n})_{n\in\N^*}$ do not vanish at $x=\ell$. On another hand, we establish that the spectral gap
"$\big{|}{\lambda_{n+1}}-{\lambda_{n}}\big{|}"$ is uniformly positive.  To this end, we consider the following spectral problem which arises by applying
separation of variables to system \eqref{bihar7},
\begin{align}\label{biharr1}\begin{cases}
               (\sigma(x)\phi'')''-(q(x)\phi')'=\lambda \rho(x) \phi,~~x\in(0,\ell),  \\
               \phi(0)=\phi'(0)=\phi(\ell)=\phi'(\ell)=0.
             \end{cases}\end{align}
It is clear that, problem
\eqref{biharr1} is equivalent to the following
spectral problem
$$\mathcal{A}\phi=\la \phi,~~\phi\in \mathcal{D}(\mathcal{A}),$$
i.e., the eigenvalues $\(\la_n\)_{n\in\N^*}$ of the operator
$\mathcal{A}$ and problem \eqref{biharr1}
coincide together with their multiplicities. One has:
\begin{Theo}\label{Lem2}
All the eigenvalues $(\la_n)_{n\in\N^*}$ of the spectral problem
\eqref{biharr1} are simple such that :
$$0<\lambda_{1}<\lambda_{2}<.......<\lambda_{n}<.....
\underset{n\rightarrow +\infty}{\longrightarrow}+\infty.$$
Moreover, the corresponding
eigenfunctions $\(\Phi_n\)_{n\in\N^*}$ satisfy \be
\Phi_n''(\ell)\not=0~~\forall n\in\N^*. \label{biharr2}\ee
\end{Theo}
Our main tool in proving this is the following result \cite[Lemma 3.2]{Hedibenamara2}.
\begin{Lem}\label{simplctyLs1m}
Let $u$ be a nontrivial solution the linear fourth
order differential equation defined on the interval $\[a,b\]$,
$a>b$: $$ (\sigma(x) u'')''-(q(x) u')'-\rho(x)
u=0, $$ where the functions $\rho(x)>0$, $\sigma(x)>0$ and $q(x)\geq0$.
 If $u, u',
u''$ and \\$\mathcal{T}u= (\si(x)u'')' - q(x)u'$ are nonnegative at $x=a$ (but
not all zero), then they are positive for all $x>a$. If $u, -u',
u''$ and $\(-\mathcal{T}u\)$ are nonnegative at $x=b$ (but not all zero), then
they are positive for all $x<b$.
\end{Lem}
\begin{proof1} {\it of Theorem \ref{Lem2}.}
First, we prove that the set $\mathcal{E}_\la$, of solutions of the following
boundary value problem
\begin{equation}\label{biharr3}
\left\{
\begin{array}{ll}
(\se(x)\phi'')'' -(q(x)\phi')'= \la \rho(x) \phi,~~x\in(0,\ell), \\
\phi(0) = \phi'(0)=\phi'(\ell)=0,
\end{array}
\right .
\end{equation}
is one-dimensional subspace for $\lambda>0$, i.e., $\Dim\,\mathcal{E}_{\la}=1$. Suppose that there exist two linearly independent
solutions $\phi_1$ and $\phi_2$ of problem \eqref{biharr3}. Both $ \phi''_{1}(0)$ and $
\phi''_{2}(0)$ must be different from zero since otherwise it would
follow from the first statement of Lemma~\ref{simplctyLs1m} that $
\phi'_i(\ell)>0~(i=1,2)$ which contradicts the last
boundary condition in \eqref{biharr3}. In view of the
assumptions about $\phi_1$ and $\phi_2$, the solution
 $$\phi(x)=  \phi''_{1}(0)\phi_{2}(x) -  \phi''_{2}(0)\phi_{1}(x)$$
  satisfies
 $$\phi(0)= \phi'(0)= \phi''(0)=0 \hbox{ and } \phi'(\ell)=0.$$ This again
 contradicts the first statement of
Lemma~\ref{simplctyLs1m} unless $\phi\equiv0$. Therefore, $$\Dim\,\mathcal{E}_{\la}=1,$$
and then,  all the  eigenvalues $\(\la_n\)_{n\in\N^*}$ of problem \eqref{biharr1} are geometrically simple. On the other hand, by
Lemma \ref{rr}, the operator $\mathcal{A}$ is self-adjoint in
$L^2_\rho(0,\ell)$, and this implies that all the eigenvalues
$\(\la_n\)_{n\in\N^*}$ are algebraically simple. Now, we prove \eqref{biharr2}. Let $\{\la_n,\Phi_{n}\}$ $(n\geq1)$
be an eigenpair of problem \eqref{biharr1}, and assume that $\Phi_{n}''\(\ell\)=0,$ for some $n\in\N^*$. Then the eigenfunctions $\Phi_{n}$  satisfy the boundary conditions  $$\Phi_{n}\(\ell\)=\Phi_{n}'\(\ell\)=\Phi_{n}''\(\ell\)=0, \hbox{ for some } n\in\N^*,$$
and then, by standard theory of differential equations $$\mathcal{T}\Phi_{n}(\ell)= (\si\(\ell\)\Phi_{n}\(\ell\)'')' - q\(\ell\)\Phi_{n}\(\ell\)'\neq0, \hbox{ for some } n\in\N^*.$$ Without loss of generality, let $\mathcal{T}\Phi_{n}\(\ell\)<0$ for some $n\in\N^*$. Since $\lambda_n>0$, it
follows from the second statement of Lemma~\ref{simplctyLs1m}, that
$$\varphi_n(x)>0,~\varphi_n'(x)<0,~\varphi_n''(x)>0 \hbox{ and } \mathcal{T}\varphi_n(x)<0,~\forall~x\in\[0,\ell\],$$ but this contradicts the boundary
conditions $\Phi_{n}\(0\)=\Phi_{n}'\(0\)=0$ . Thus, $$\Phi_n''(\ell)\not=0~~ \forall n\in\N^*,$$ and this finalizes the proof of the
theorem.
\end{proof1}\\
\indent Next we establishes the asymptotic
behavior of the spectral gap ${\la_{n+1}}-{\la_{n}}$ for
large $n$. Namely, we have the following theorem:
\begin{Theo}\label{SP2}  The eigenvalues $(\la_{n})_{n\in\N^*}$ of the
associated spectral problem \eqref{biharr1} satisfy the
following asymptotic: \be \sqrt[4]{{\la_n}}:=\mu_n= \frac{\pi}{\ga}\(n-\frac{1}{2}\) +\mathcal{O}\(\frac{1}{\exp\(n\)}\),~\gamma={\int^\ell_0
\sqrt[4]{\frac{\rho(t)}{\sigma(t)}}
 dt}.\label{biharr4}\ee
Moreover,
 \be |{\la_{n+1}}-{\la_n}|\asymp {n^3}\(\frac{\pi}{\gamma}\)^{4},~\hbox{ as } n\to\infty.\label{biharr5}\ee
\end{Theo}
\begin{proof} It is known (e.g., \cite[Chapter~5, p.235-239]{F} and
\cite[Chapter~2]{N}) that for $\la\in\mathbb{C}$, the fourth-order linear differential equation \be \label{biharr6}
(\se(x)\phi'')'' -(q(x)\phi')'= \la \rho(x) \phi,~~x\in(0,\ell),\ee 
has four fundamental solutions $\{{\phi}_i(x,\lambda)\}_{i=1}^{i=4}$ satisfying the asymptotic forms 
\bean
\label{biharr7}\left\{
\begin{array}{ll} {\phi}_i(x,\la)
=\(\[\rho(x)\]^{\frac{3}{4}}\[\sigma(x)\]^{\frac{1}{4}}\)^{-\frac{1}{2}}\exp\left\{\mu
w_i \displaystyle\int^x_0 \sqrt[4]{\frac{\rho(t)}{\sigma(t)}}
 dt\right\} [1],\\
~~\\
{\phi}_{i}^{(k)}(x,\la) = (\mu w_i)^{k}
\(\frac{\rho(x)}{\sigma(x)}\)^{\frac{k}{4}}
\(\[\rho(x)\]^{\frac{3}{4}}\[\sigma(x)\]^{\frac{1}{4}}\)^{-\frac{1}{2}}
\exp\left\{\mu w_i\displaystyle \int^x_0 \sqrt[4]{\frac{\rho(t)}{\sigma(t)}}
 dt\right\} [1],\end{array}
\right .\eean where ${\mu}^4=\la$, ${w_i}^4=1$,
${\phi}^{(k)}:=\frac{\partial^k{\phi}}{{\partial x^k}}$ for
$k\in\{1, 2, 3\}$, and
 $\[1\]=1 + \mathcal{O}(\mu^{-1})$ uniformly as $\mu\rightarrow\infty$ in a sector
$\mathcal{S}_\tau=\{\mu\in\C \hbox{~~such
that~~}0\leq\arg(\mu+\tau)\leq\frac{\pi}{4}\}$ where $\tau$ is any
fixed complex number. It is convenient to rewrite these asymptotes
in the form\bea
&&{\phi}_1(x,\la)=\(\[\rho(x)\]^{\frac{3}{4}}\[\sigma(x)\]^{\frac{1}{4}}\)^{-\frac{1}{2}}\cos\(\mu \displaystyle \int^x_0 \sqrt[4]{\frac{\rho(t)}{\sigma(t)}}
 dt\)[1],\\
&&{\phi}_2(x,\la)=\(\[\rho(x)\]^{\frac{3}{4}}\[\sigma(x)\]^{\frac{1}{4}}\)^{-\frac{1}{2}}\cosh\(\mu \displaystyle \int^x_0 \sqrt[4]{\frac{\rho(t)}{\sigma(t)}}
 dt\)[1],\\
&&{\phi}_3(x,\la)=\(\[\rho(x)\]^{\frac{3}{4}}\[\sigma(x)\]^{\frac{1}{4}}\)^{-\frac{1}{2}}\sin\(\mu \displaystyle \int^x_0 \sqrt[4]{\frac{\rho(t)}{\sigma(t)}}
 dt\)[1],\\
&&{\phi}_4(x,\la)=\(\[\rho(x)\]^{\frac{3}{4}}\[\sigma(x)\]^{\frac{1}{4}}\)^{-\frac{1}{2}}\sinh\(\mu \displaystyle \int^x_0 \sqrt[4]{\frac{\rho(t)}{\sigma(t)}}
 dt\)[1]. \eea Hence every solution $\phi(x,\la)$ of equation
\eqref{biharr6} can be written in the following asymptotic form \bean
\phi(x,\la)=\zeta(x)\Big(C_1 \cos\(\mu X\) + C_2 \cosh\(\mu X\)+ C_3
\sin\(\mu X\)
+ C_4 \sinh\(\mu X\)\Big)\[1\]\label{biharr8}
\eean and from \eqref{biharr7}, we have also  \bean
\phi^{(k)}(x,\la)&=&\mu^k\zeta(x)
\(\frac{\rho(x)}{\sigma(x)}\)^{\frac{k}{4}}\Big{(}C_1 \cos^{(k)}(\mu
X) + C_2
\cosh^{(k)}(\mu X) + C_3 \sin^{(k)}(\mu X) \nonumber\\
 &&+ C_4 \sinh^{(k)}(\mu X)\Big{)}\[1\],~~\mbox{as}~~\mu \rightarrow\infty,~~
k\in\{1, 2, 3\}, \label{biharr9} \eean
 where $C_i, i=1,2,3,4$ are constants and \be
\zeta(x)=\(\[\rho(x)\]^{\frac{3}{4}}\[\sigma(x)\]^{\frac{1}{4}}\)^{-\frac{1}{2}}
\hbox{~~and~~} X=\int^x_0 \sqrt[4]{\frac{\rho(t)}{\sigma(t)}}
 dt.\label{biharr10}\ee If $\phi(x,\la)$ satisfies
the boundary conditions $\phi(0,\la)=\phi'(0,\la)=0$, then by the asymptotics
\eqref{biharr8} and \eqref{biharr9}, we obtain for large positive $\mu$
the asymptotic estimate
$$\begin{cases}
\zeta(0)\(C_1+C_2\)[1]=0, \\
\mu\zeta(0)
\(\frac{\rho(0)}{\sigma(0)}\)^{\frac{1}{4}}\(C_3+C_4\)[1]=0.
                        \end{cases}$$
and then,
\bean \phi(x, \la)=C_1
\zeta(x)\(\cos\(\mu X\)- \cosh\(\mu X\)\)\[1\]+ C_3\(
\sin\(\mu X\)-\sinh\(\mu X\)\)\[1\]
\label{biharr11}
\eean and \bean \phi'(x, \la) = {\mu
\zeta(x)}\(\frac{\rho(x)}{\sigma(x)}\)^{\frac{1}{4}}
          \Big(C_1\(\sinh\(\mu X\)-\sin\(\mu X\)\)+ C_3\(
\cos\(\mu X\)-\cosh\(\mu X\)\)\Big)\[1\],\label{biharr12}
\eean
From the boundary conditions $\phi(\ell,\lambda)=\phi'(\ell,\lambda)=0$, and the above asymptotics one has:
\bean\begin{cases}
C_1 \(\cos\(\mu \ga\)- \cosh\(\mu \ga\)\)\[1\]+ C_3\(
\sin\(\mu \ga\)-\sinh\(\mu \ga\)\)\[1\]=0, \\
C_1\(-\sin\(\mu \ga\)-\sinh\(\mu \ga\)\)\[1\]+ C_3\(
\cos\(\mu \ga\)-\cosh\(\mu \ga\)\)\[1\]=0,
                        \end{cases}\label{biharr13}\eean
where the constant $\ga$ is defined by \be \ga=\int^l_0
\sqrt[4]{\frac{\rho(t)}{\sigma(t)}}
 dt.\label{biharr14}\ee
This homogeneous system of equations in the unknowns $C_1$ and $C_2$ admits a non-trivial solution if and only if the corresponding determinant is zero, i.e., $$\(
\(\cos\(\mu \ga\)-\cosh\(\mu \ga\)\)^2+\sin^2\(\mu \ga\)-\sinh^2\(\mu \ga\)\)\[1\]=0$$
Equivalently $${\mu
\zeta(\ell)}\(\frac{\rho(\ell)}{\sigma(\ell)}\)^{\frac{1}{4}}\(\cos\(\mu \ga\)\cosh\(\mu \ga\)-1\)\[1\]=0.$$
Then by \eqref{biharr12}, one gets that the eigenvalues $\(\la_n\)_{n\in\N^*}$  are solution of following asymptotic characteristic equation
 $${\mu
\zeta(\ell)}\(\frac{\rho(\ell)}{\sigma(\ell)}\)^{\frac{1}{4}}\exp\(\mu \ga\)\(\cos\(\mu \ga\)-\frac{1}{\exp\(\mu \ga\)}\) \[1\]=0,$$
 which can also be rewritten as
\be \cos\(\mu \ga\)+\mathcal{O}\(\frac{1}{\exp\(\mu \ga\)}\)=0.\label{biharr15}\ee Since
the solutions of the equation $\cos\(\mu \ga\) = 0$ are given by
$$\widetilde{\mu_{n}}=\frac{\pi}{\ga}\(n-\frac{1}{2}\),~n=0,1,2,..., $$ it follows
from Rouch\'{e}'s theorem that the solutions of \eqref{biharr15} satisfy
the following asymptotic \bean \mu_{n}&=& \widetilde{\mu_{n}} +
\delta_n \nonumber\\
&=& \frac{\pi}{\ga}\(n-\frac{1}{2}\) +\mathcal{O}\(\frac{1}{\exp\(n\)}\),\label{biharr16}\eean which proves \eqref{biharr4}.
Furthermore, \bea \sqrt{\la_n}&=&
\frac{\pi^2}{\ga^2}\(n-\frac{1}{2}\)^2 +\mathcal{O}\(\frac{n}{\exp\(n\)}\),\\
&=&\frac{\pi^2}{\ga^2}\( n^2-n\)+\mathcal{O}\(1\).\eea
 and hence
\begin{align*}
  {\la_{n+1}}-{\la_n} & = \(\sqrt{\la_{n+1}}-\sqrt{\la_n} \) \(\sqrt{\la_{n+1}}+\sqrt{\la_n} \)\\
  &= \frac{\pi^4}{\ga^4}\(\(n+1\)^2-n^2+\mathcal{O}\(1\)\)\(\(n+1\)^2+n^2-2n+\mathcal{O}\(1\)\)\\
  &= \frac{\pi^4}{\ga^4}n^3+\mathcal{O}\(n^2\).
\end{align*}
The theorem is proved.
\end{proof}
We conclude this section with the following result about the asymptotics of the eigenfunctions $(\Phi_{n})_{n\in\N^*}$ of the spectral problem \eqref{biharr1}.
\begin{Prop}\label{Lem22}
Let us normalize the eigenfunctions $(\Phi_{n})_{n\in\N^*}$ of the spectral problem \eqref{biharr1} in the sense that
$\displaystyle\lim_{n\to\infty}\|{\Phi}_n\|_{L^2_\rho(0,\ell)}=1$. One has, the following asymptotic estimates:
\begin{align}
 \Phi_{n}(x)&=\frac{2\zeta(x)}{\ga\exp{\(\mu_n \ga\)}}
\(\cos\(\mu_n \ga\)- \cosh\(\mu_n \ga\)\)\(\cos\(\mu_n X\)- \cosh\(\mu_n X\)\)\[1\]\nonumber\\
&~+\frac{2\zeta(x)}{\ga\exp{\(\mu_n \ga\)}}\(\sin\(\mu_n \ga\)+\sinh\(\mu_n \ga\)\)\(
\sin\(\mu_n X\)-\sinh\(\mu_n X\)\)\[1\], \label{biharr17}
\end{align}
where the quantities $\zeta$, $X$  and $\gamma$ are given by \eqref{biharr10} and \eqref{biharr14}, respectively.
Furthermore,  \be \displaystyle\lim_{n\rightarrow\infty} \dfrac{\left|
{\Phi_{n}''(\ell)}\right|}{\sqrt{\la_n}}=
\(\frac{2\zeta(\ell)}{\ga}\(\frac{\rho(\ell)}{\sigma(\ell)}\)^{\frac{1}{2}}\).\label{biharr18}\ee
\begin{equation*}
\end{equation*}
\end{Prop}
\begin{proof} If $\mu_n$ satisfies \eqref{biharr15}, then, by solving the homogeneous system of two equations \eqref{biharr13},
one gets \bean\begin{cases}
C_1 = C\(\cos\(\mu_n \ga\)- \cosh\(\mu_n \ga\)\)\[1\] \\
C_3= C\(\sin\(\mu_n \ga\)+\sinh\(\mu_n \ga\)\)\[1\].
                        \end{cases}\label{biharr19}\eean
for some constant $C\not=0$. From this, \eqref{biharr4} and \eqref{biharr12},
we obtain the following asymptotic estimate for the eigenfunctions $\phi(x,\lambda_n)$ of the problem
\eqref{biharr1}:\begin{align}
                      \phi(x,\lambda_n)&= C\zeta(x)\{\(\cos\(\mu_n \ga\)- \cosh\(\mu_n \ga\)\)\(\cos\(\mu_n X\)- \cosh\(\mu_n X\)\)\}\[1\]\nonumber\\
                      &+C\zeta(x)\{\(\sin\(\mu_n \ga\)+\sinh\(\mu_n \ga\)\)\(
\sin\(\mu_n X\)-\sinh\(\mu_n X\)\)\}\[1\].
\label{biharr11bis}
                    \end{align}
By \eqref{biharr16} and \eqref{biharr19} $$
C_1 \asymp\frac{\exp{\(\mu_n \ga\)}}{2}\asymp C_3,
                        $$
and then,
\bea \Phi_{n}(x)&\sim& C\zeta(x)\frac{\exp{\(\mu_n \ga\)}}{2}\(\sin\(\mu_n X\)-\cos\(\mu_n X\)
+\cosh\(\mu_n X\)-\sinh\(\mu_n X\)\)\\
&\sim& C\zeta(x)\frac{\exp{\(\mu_n \ga\)}}{2}\(\sin\(\mu_n X\)-\cos\(\mu_n X\)\),~\hbox{ as }
n \rightarrow\infty,\eea
where $\ga$ is defined by \eqref{biharr14}. By the change of variables $t=X$, one has
\begin{align*}
\int_{0}^{\ell}\xi^2(x)\sin^2\(\mu_nX\)\rho(x)dx&= \int_{0}^{\ell}\sin^2\(\mu_n\int^x_0
\sqrt[4]{\frac{\rho(t)}{\sigma(t)}}
 dt\)\sqrt[4]{\frac{\rho(x)}{\sigma(x)}}dx, \\
&=\int_{0}^{\gamma}\sin^2\(\mu_nt\)dt=\dfrac{\gamma}{2}.
\end{align*}
Similarly, we have \begin{align*}
                     &\int_{0}^{\ell}\xi^2(x)\cos^2\(\mu_nX\)\rho(x)dx=\dfrac{\gamma}{2}, \\
                     &\int_{0}^{\ell}\xi^2(x)\sin\(\mu_nX\)\cos\(\mu_nX\)\rho(x)dx=
\frac{\sin^2(\mu_n\gamma)}{2\mu_n}[1].
                   \end{align*}
Consequently, one gets \be \displaystyle\lim_{n\to\infty}\left\|\phi(x,\lambda_n)\right\|_{L^2_\rho(0,l)}=|C|
\frac{\ga\exp{\(\mu_n \ga\)}}{2}.\label{biharr20}\ee
We set \be \Phi_{n}(x):=\frac{\phi(x,\lambda_n)}
{\displaystyle\lim_{n\to\infty}\left\|\phi(x,\lambda_n)\right\|_{L^2_\rho(0,l)}}.\label{biharr21}\ee Then, $\(\Phi_{n}(x)\)_{n\in\N^*}$ are the normalized eigenfunctions of problem \eqref{biharr1} so that, $\displaystyle\lim_{n\to\infty}\left\|\Phi_{n}\right\|_{L^2_\rho(0,l)}=1$. Therefore, by \eqref{biharr11bis} and \eqref{biharr20}-\eqref{biharr21}, we get \eqref{biharr17}.

\medskip

In a similar way, from the asymptotics \eqref{biharr9}, \eqref{biharr16} and
\eqref{biharr20}, a straightforward computation yields
\begin{align*}
\Phi_{n}''(x)=&\frac{-2\mu^2\zeta(x)}{\ga\exp{\(\mu_n \ga\)}}\(\frac{\rho(x)}{\sigma(x)}\)^{\frac{1}{2}}
\(\cos\(\mu_n \ga\)- \cosh\(\mu_n \ga\)\)\(\cos\(\mu_n X\)+\cosh\(\mu_n X\)\)\[1\]\\&-\frac{2\mu^2\zeta(x)}{\ga\exp{\(\mu_n \ga\)}}\(\frac{\rho(x)}{\sigma(x)}\)^{\frac{1}{2}}\(\sin\(\mu_n \ga\)+\sinh\(\mu_n \ga\)\)\(
\sin\(\mu_n X\)+\sinh\(\mu_n X\)\)\[1\].
                                                  \end{align*}
As consequence, one has
\bea |\Phi_{n}''(\ell)|=\frac{4\mu_n^2\zeta(\ell)}{\ga\exp{\(\mu_n \ga\)}}\(\frac{\rho(\ell)}{\sigma(\ell)}\)^{\frac{1}{2}}\left|\sin\(\mu_n \ga\)\sinh\(\mu_n \ga\)\right|[1]. \eea
Therefore, from this and the asymptote \eqref{biharr4}, we get \eqref{biharr18}.  The proof is complete.
\end{proof}
\section{Observability}
In this section, we prove some observability results which are consequences of the asymptotic properties of the previous section. The reason to study these properties is that, by means of the Lions'{\rm HUM}
\cite{J.L2}, controllability properties can be reduced to suitable observability
inequalities for the adjoint system. As \eqref{bihar2} is a self-adjoint system, we are reduced to the same
system, without control. Therefore, consider system \eqref{bihar2} without control, i.e.,
\bean
\label{biharr22} \left\{
 \begin{array}{ll}
i\rho(x)\partial_t\tilde y=-\partial_x^2\(\sigma(x)\partial_x^2\tilde y\)+
\partial_x(q(x)\partial_x\tilde y)_x,&(t,x)\in(0,T)\times(0,\ell),\\
\tilde y(t, 0) =\partial_x\tilde  y(t,0) =\tilde  y(t, \ell) =\partial_x\tilde y(t, \ell) = 0,&t\in(0,T),\\
\tilde y(0, x) = \tilde  y^{0},& x\in(0,\ell).
\end{array}
 \right.
\eean
One has:
\begin{Prop}\label{ph1}
Let $T
>0$ and  $\tilde y^0\in H^2_0(0,\ell)$. Then
\begin{equation}\label{biharr23}
\int_{0}^{T}|\partial_x^2\tilde{y}(t,\ell)|^{2}dt \asymp
\|\tilde{y}^0\|_{H^2_0(0,\ell)}^{2},
\end{equation}
where $\tilde y$ is the solution of problem \eqref{biharr22}.
\end{Prop}
In order to prove Proposition \ref{ph1}, we need the following
variant of Ingham's inequality due to Beurling (e.g., \cite{L}).
\begin{Lem}\cite{L}\label{In}
Let $({\la}_n)_{n\in\mathbb{Z}}$  be a strictly increasing sequence satisfying for some $\delta > 0$
the condition $$
|{\la}_{n+1}-{\la}_n|>\delta,~\forall~ n\in\mathbb{Z}.$$
Then, for any $T>2\pi D^+\(\la_n\)$, the family $\(e^{i\lambda_n t}\)_{n\in\mathbb{Z}}$ forms a Riesz basis in $L^2(0, T)$, that is
\begin{equation*}
\int_{0}^{T}\left|\sum\limits_{n\in\mathbb{Z}}c_ne^{i{\la}_nt}\right|^{2}dt \asymp\sum\limits_{n\in
\mathbb{Z}}|c_{n}|^{2},
\end{equation*}
where $ D^+\(\la_n\):= \displaystyle \lim_{r\to\infty}
\frac{n^+\(r,\lambda_n\)}{r}$ is the Beurling upper density of the sequence $(\lambda_{n})_{n\in\N^*},$ with $n^+\(r ,\lambda_n\)$ denotes the maximum number of terms of the sequence $(\lambda_{n})_{n\in\N^*}$ contained
in an interval of length $r$.
\end{Lem}
\begin{proof1} {\it of Proposition \ref{ph1}.} It follows
from the spectral representation \eqref{bihar8} of the space $\mathcal{H}_{\theta},$
that\begin{align*}
      \mathcal{H}_{1/2}&=
\left\{u(x)=\sum\limits_{n\in\N^*}c_n\Phi_n(x)~:
 ~\|u\|_{\theta}^2=\sum\limits_{n\in\N^*}\la_n|c_n|^{2}<\infty\right\} \\
       &=\mathcal D\(\mathcal{A}^{1/2}\)=H_0^2\(0,\ell\),
    \end{align*}
where the eigenfunctions $\(\Phi_{n}\)_{n\in\N^*}$ are given in
Proposition \ref{Lem22}. By Proposition \ref{pos}, the solution $\tilde y$  of problem \eqref{biharr22} has the form $$
\tilde{y}(t,x)=\sum\limits_{n\in\N^*}c_ne^{i\la_nt}{\Phi}_n(x),$$
where $\tilde {y}^0=\sum\limits_{n\in\N^*} c_n\phi_n$. Consequently,  \be
\int_{0}^{T}|\partial_x^2\tilde{y}(t,\ell)|^{2}dt =
\int_{0}^{T}\Big{|}\sum\limits_{n\in \N^*}
c_{n}e^{i{\la}_n t} \Phi_n''\(\ell\)\Big{|}^{2}dt. \label{biharr24}\ee  Thus by the first statement of Theorem
\ref{Lem2} and the gap condition \eqref{biharr5}, Beurling's Lemma \ref{In}
states that for any $T>D^+\(\la_n\)$, the family $\(e^{i\lambda n
t}\)_{n\in\N^*}$ forms a Riesz basis in $L^2(0, T)$, where $D^+\(\la_n\)$ is the Beurling upper density of the eigenvalues
$(\lambda_{n})_{n\in\N^*}$. Furthermore,
for every $T>D^+\(\la_n\),$ one has \be \int_{0}^{T}\Big{|}\sum\limits_{n\in \N^*}
c_{n}e^{i{\la}_n t} \Phi_n''\(\ell\)\Big{|}^{2}dt\asymp \sum\limits_{n\in
\N^*}\left|c_n \Phi''_n(\ell)\right|^{2}.\label{biharr25}\ee
From  the asymptote \eqref{biharr4} and the characteristic equation \eqref{biharr15}, we find that the Beurling upper density of the eigenvalues
$(\lambda_{n})_{n\in\N^*}$,
$$ D^+\(\la_n\)=\lim_{n\to\infty}\frac{{\ga}^4}{{\pi^4}\(n-\frac{1}{2}\)^3}=0.$$
By the second statement of Theorem \ref{Lem2}, we have
$$\Phi'_n(\ell)\not= 0 \hbox{ for all } n\in \mathbb{N}^*,$$
and then by \eqref{biharr18}, we deduce that there exists $C_1,~ C_2 > 0$ such that $$
C_1{\la_n}\leq\left|\Phi''_n(\ell)\right|^2\leq C_2{\la_n},~\hbox{ as } n\to\infty.$$ Therefore from the above and \eqref{biharr25}, for any $T>0$
$$\int_{0}^{T}\Big{|}\sum\limits_{n\in \N^*}
c_{n}e^{i{\la}_n t} \Phi_n''\(\ell\)\Big{|}^{2}dt\asymp \sum\limits_{n\in
\N^*}\lambda_n\left|c_n\right|^{2}.$$
Thus from this and \eqref{biharr24}, we get \eqref{biharr23}.
 This completes the proof.
\end{proof1}
\section{Exact boundary controllability}
In this section, we prove the exact boundary controllability of the control problem  \eqref{bihar2}.
\subsection{Well-posedness} Since we are dealing with boundary control, we
need to introduce the weaker notion of "solution defined by transposition" in the spirit of \cite{V.K, J.L3}.

\medskip

Let $\tilde y$ be the solution to problem \eqref{biharr22} satisfying \eqref{bihar9}. Now let $f\in C^{\infty}(0,T)$ (or $f\in L^{2}(0,T)$ since $C^{\infty}(0,T)$ is dense in $L^{2}(0,T)$) and let $y\in C^{4}\([0,T];(0,\ell)\)$ be a function satisfying \eqref{bihar2}. Then we multiply \eqref{bihar7} by $y$ and integrate on $(0,T)\times(0,\ell)$ to obtain
\bea i\int_{0}^{\ell}\int_{0}^{T}\partial_t\tilde yy(t,x)dt\rho(x)dx+\int_{0}^{\ell}\int_{0}^{T}\(\partial_x^2\(\sigma(x)\partial_x^2\tilde y\)-\partial_x(q(x)\partial_x\tilde y)_x\)y(t,x)dtdx=0.\eea
Then integrate by parts and using the boundary conditions in \eqref{bihar2} and \eqref{bihar7}, we get
\begin{align*}
i\int_{0}^{\ell}\bigg[\tilde y y(t,x)\bigg]_{0}^{T}\rho(x)dx
= \sigma(\ell)\int_{0}^{T}\partial_x^2\tilde y\(t,\ell\)f(t)dt+i\int_{0}^{\ell}\int_{0}^{T}\partial_ty\tilde y(t,x)dt\rho(x)dx\\
~~-\int_{0}^{\ell}\int_{0}^{T}\(\partial_x^2\(\sigma(x)\partial_x^2 y\)-\partial_x(q(x)\partial_x y)_x\)\tilde y(t,x)dtdx
\end{align*}
and then
\be i\int_{0}^{\ell}\tilde y y(T,x)\rho(x)dx=\sigma(\ell)\int_{0}^{T}\partial_x^2\tilde y\(t,\ell\)f(t)dt+ i\int_{0}^{\ell}\tilde y^0 y^0\rho(x)dx.\label{biharr29}
\ee
Let us define the spaces
$$\mathcal{S}:=H^{2}_0(0,\ell)\hbox{ and } \mathcal{S}':=H^{-2}(0,\ell),$$
and the linear functional $\mathcal{L}_{T}$ on $S$ by
\be \displaystyle \mathcal{L}_{T}(\tilde y^0)=i\langle y^0,\tilde y^0
\rangle_{\mathcal{S}^{'},\mathcal{S}}+\sigma(\ell)\int_{0}^{T}\partial_x^2\tilde y\(t,\ell\)f(t)dt.\label{biharr30}\ee
Moreover, we have
\be \Vert \mathcal{L}_{T}\Vert\leq C\left(\Vert y^0\Vert_{H^{-2}(0,\ell)}+\Vert f\Vert_{L^{2}(0,T)}\right).\label{biharr32}\ee
Using \eqref{biharr29}, we may rewrite the identity \eqref{biharr30} in the following form \be \displaystyle \mathcal{L}_{T}(\tilde y^0)=i\langle y(T,x),\tilde y(T,x)
\rangle_{\mathcal{S}^{'},\mathcal{S}}.\label{biharr33}\ee
This motivates the following definition.
\begin{Rem1}
We say that $y$ is a weak solution to problem \eqref{bihar2} in the sense of transposition if $y\in C\([0, T ]; H^{-2}(0,\ell)\)$ satisfies \eqref{biharr33} for all $T>0$ and for every $\tilde y^0\in \mathcal S.$
\end{Rem1}Then we have the following:
\begin{Prop} \label{b-p} Let
$T>0$, and  $f\in L^{2}(0,T)$. Then for any $y^0\in H^{-2}(0,\ell)$, there exists a unique weak solution  of
system \eqref{bihar2} in the sense of transposition, satisfying
\be y\in C\([0, T ]; H^{-2}(0,\ell)\).\label{biharr43}\ee
Moreover, there exists a constant $C>0$ such that
\be \|y\|_{L^\infty\([0, T ]; H^{-2}(0,\ell)\)}\leq
C\(\|y^0\|_{H^{-2}(0,\ell)}+\|f\|_{L^{2}(0,T)}\).\label{biharr44}\ee\end{Prop}
\begin{proof} It follows from Proposition \ref{pos}, that for any $T>0$ the linear map
$$\tilde y(T,.)\longmapsto \tilde y^0$$ is an isomorphism from
$H^{2}_0(0,\ell)$ into itself. Hence, by Proposition \ref{ph1} we deduce  that the linear map $$\tilde y(T,.)\longmapsto \mathcal{L}_{T}(\tilde y^0)$$ is continuous on
$H^{2}_0(0,\ell)$. Therefore, by duality, Equation \eqref{biharr33} defines $y(T,x)$, as
a unique element in $H^{-2}(0,\ell)$. Moreover from \eqref{biharr32} it follows that \eqref{biharr44} holds. The continuity with respect to time in \eqref{biharr43} is proved by density argument. The proof is complete.
\end{proof}
\subsection{Exact controllability}
We are now ready to state our main controllability result. Thanks to the reversibility in time of \eqref{bihar2}, this system is exactly controllable
if and only if the system is null controllable. One has:
\begin{Theo}
\label{hj} Assume that the coefficients $\rho$, $\sigma$ and $q$
satisfy \eqref{bihar3} and \eqref{bihar4}. Given $T>0$ and
 $y^0\in H^{-2}\(0,\ell\) $, there
exists a control $f\in L^{2}(0,T)$ such that the solution $y$ of the
control problem \eqref{bihar2}
 satisfies
\begin{equation*}
y(T,x)=0,~~x\in\[0,\ell\].
\end{equation*}
\end{Theo}
\begin{proof}
By the Lions'{\rm HUM} \cite{J.L2}, solving the exact controllability problem is equivalent to proving an observability
inequality for the backward problem. The backward problem is
\bean
\label{biharr31} \left\{
 \begin{array}{ll}
i\rho(x)\partial_ty=-\partial_x^2\(\sigma(x)\partial_x^2y\)+
\partial_x(q(x)\partial_xy)_x,&(t,x)\in(0,T)\times(0,\ell),\\
y(t, 0) =\partial_xy(t,0) =y(t, \ell) = 0,~\partial_xy(t, \ell) = \partial_x^2\tilde{y}(t,\ell),&t\in(0,T),\\
y(T, x)=0,& x\in(0,\ell),
\end{array}
 \right.
\eean
where $\tilde y$ is the solution of the uncontrolled system  \eqref{biharr22}. By Proposition \ref{b-p}, problem \eqref{biharr31} has a unique weak solution $y$, satisfying
$y^0:=y(0,x) \in H^{-2}(0,\ell).$ Hence the linear map
$$\Lambda~:~H^{2}_0(0,\ell)\longrightarrow H^{-2}(0,\ell),~~ \tilde y^0\longmapsto -iy^0$$ is continuous from
$H^{2}_0(0,\ell)$ into $H^{-2}(0,\ell)$. Furthermore, if $\Lambda$ is shown to be surjective then there exists a control of the form $f(t) = \partial_x^2\tilde{y}(t,\ell)$ which drives the system \eqref{bihar2} to rest in time
$T.$ Since $y(T, x)=0$, then for the choice of $f(t) = \partial_x^2\tilde{y}(t,\ell)$ by \eqref{biharr30}, one has
$$-i\langle y^0,\tilde y^0
\rangle_{\mathcal{S}^{'},\mathcal{S}}=\sigma(\ell)\int_{0}^{T}\partial_x^2\tilde y\(t,\ell\)dt.$$
Equivalently $$\langle \Lambda  (\tilde y^0),\tilde y^0
\rangle_{\mathcal{S}^{'},\mathcal{S}}=\sigma(\ell)\int_{0}^{T}\partial_x^2\tilde y\(t,\ell\)dt.$$
By Proposition \ref{ph1}, for every $T
>0$ and  $\tilde y^0\in H^2_0(0,\ell)$, we have
$$
\int_{0}^{T}|\partial_x^2\tilde{y}(t,\ell)|^{2}dt \asymp
\|\tilde{y}^0\|_{H^2_0(0,\ell)}^{2}.
$$
Consequently from the above, for every $T
>0,$  $$\langle \Lambda  (\tilde y^0),\tilde y^0
\rangle_{\mathcal{S}^{'},\mathcal{S}}\asymp
\|\tilde{y}^0\|_{H^2_0(0,\ell)}^{2}.$$
Therefore by the Lax–Milgram Theorem, $\Lambda$ is surjective.  This implies that there exists a control of the form $f(t) = \partial_x^2\tilde{y}(t,\ell)$ which drives the system \eqref{bihar2} to rest in time
$T>0,$ and this completes the proof of Theorem \ref{hj}.
\end{proof}

\end{document}